\documentclass[12pt,a4paper,twoside]{article}

\usepackage{fancyhdr}

\usepackage{latexsym}
\usepackage{amsmath}
\usepackage{verbatim}
\usepackage{amsthm}
\usepackage{amssymb}
\def\R{{\ifmmode{\rm I}\mkern-4mu{\rm R}
\else\leavevmode\hbox{j}\kern-.17em\hbox{R}\fi}}
\def\N{{\ifmmode{\rm I}\mkern-3.5mu{\rm N}
\else\leavevmode\hbox{j}\kern-.16em \hbox{N}\fi} }
\def\C{\ifmmode{{\rm C}\mkern-15mu{\phantom{\rm t}\vrule}}\mkern10mu
\else\leavevmode\hbox{C}\kern-.5em\hbox{j}\kern.3em\fi}

\def\Q{{\ifmmode{\rm I}\mkern-7.5mu{\rm Q}
\else\leavevmode\hbox{j}\kern-.17em\hbox{Q}\fi}}

\newcommand{\sfrac}[2]{\mbox{$\frac{#1}{#2}$}}
\def\1{{\mathchoice {1\mskip-4mu\mathrm l}      
{1\mskip-4mu\mathrm l} 
{1\mskip-4.5mu\mathrm l} {1\mskip-5mu\mathrm l}}} 
\newcommand{\ssup}[1] {{{\scriptscriptstyle{({#1}})}}} 

\renewcommand{\d}{{\rm d}}
\newcommand{\e}{{\rm e}}
\renewcommand{\phi}{\varphi}

\newcommand{\heap}[2]{\genfrac{}{}{0pt}{}{#1}{#2}} 
\newcommand{\eps}{\varepsilon}
\renewcommand{\epsilon}{\varepsilon}

\newcommand{\sign}{{\rm sign}}

\renewcommand{\P}{\mathbb{P}}

\newcounter{numerator}%

\newtheorem{theorem}{Theorem}[section]
\newtheorem{corollary}[theorem]{Corollary}
\newtheorem{proposition}[theorem]{Proposition}
\newtheorem{lemma}[theorem]{Lemma}
\newtheorem{remark}[theorem]{Remark}

\newcommand{\Sym}{\mathfrak{S}}

\begin{document}

\title{Brownian motion in a truncated Weyl chamber}

\author{Wolfgang K\"onig and Patrick Schmid}

\date{August 18, 2010}

\maketitle

\abstract{We examine the non-exit probability of a multidimensional Brownian motion from a growing truncated Weyl chamber. Different regimes are identified according to the growth speed, ranging from polynomial decay over stretched-exponential to exponential decay. Furthermore we derive associated large deviation principles for the empirical measure of the properly rescaled and transformed Brownian motion as the dimension grows to infinity. Our main tool is an explicit eigenvalue expansion for the transition probabilities before exiting the truncated Weyl chamber.}

\bigskip


\medskip


\bigskip

\section{Introduction}\label{sec-Intro}

\noindent Our goal is to examine the non-exit probability of a Brownian motion from a growing truncated Weyl chamber for long times. Let $k\in\N$ be fixed and let $B=(B(t))_{t\in[0,\infty)}$ be a standard Brownian motion in $\R^k$. Furthermore, let $W=W_A=\{x = (x_1, \ldots, x_k) \in \mathbb{R}^k \colon x_1 < \ldots < x_k \}$ be the Weyl chamber of type A. Then it is well-known \cite{Gra} that the asymptotics of the probability not to exit $W$ for a long time is given by
\begin{equation}\label{Wasy}
 \P_x(B_{[0,t]}\subset W)\sim K h(x)t^{-\frac k4(k-1)},\qquad t\to\infty, \mbox{ for }x\in W,
\end{equation}
where the motion starts from $x\in\R^k$ under $\P_x$, $K$ is an explicit constant, and
\begin{equation}\label{Vandermonde}
h(x)=\prod_{1\leq i<j\leq k}(x_j-x_i)=\det\big[(x_i^{j-1})_{i,j=1,\dots,k}\big]
\end{equation}
denotes the well-known {\it Vandermonde determinant}. On the other hand, it is also well-known, see \cite{Stone} for example, that the non-exit probability from the bounded set $W\cap I^k$ with $I=(-\frac \pi2,\frac\pi 2)$ is asymptotically given as
\begin{equation}\label{Wboxasy}
\P_x(B_{[0,t]}\subset W\cap I^k)\sim \e^{-t\lambda^{\ssup{W\cap I^k}}} f^{\ssup{W\cap I^k}}(x)\langle f^{\ssup{W\cap I^k}},\1\rangle,\qquad t\to\infty, \mbox{ for }x\in W,
\end{equation}
where $\lambda^{\ssup U}$ denotes the principal eigenvalue and $f^{\ssup U}$ the corresponding positive $L^2$-normalised eigenfunction of $-\frac 12\Delta$ in an open bounded connected set $U\subset\R^k$ with Dirichlet (i.e., zero) boundary condition, and $\langle f,g\rangle $ denotes the standard inner product in $L^2(U)$. That is, the probability of not exiting from the Weyl chamber decays polynomially in time, while the one for the truncated Weyl chamber decays even exponentially.

The first main goal of this paper is to understand the transition from exponential to polynomial decay when replacing the box $I^k$ by the time-dependent box $r(t) I^k$ for different choices of a growth function $r\colon(1,\infty)\to (0,\infty)$. In particular, an interesting question is how the two functions $h$ and $f^{\ssup{W\cap I^k}}$ are transformed into each other. Is it true that the Vandermonde determinant is equal to a rescaled limit of the principal eigenfunction of $-\frac 12 \Delta$ in $W\cap I^k$?

It will turn out that, for $1\ll r(t)\ll \sqrt t$, the non-exit probability decays in a stretched-exponential way, but for $\sqrt t\ll r(t)$, the same asymptotics as in \eqref{Wasy} will hold, since the motion does not feel the boundary, according to the central limit theorem. However, the way in which the stretched-exponential decay becomes a polynomial decay when $r(t)\asymp \sqrt t$, is {\it a priori} not clear. This is one of the main topics of this paper. Here is a short version of our main result on this (see Theorem~\ref{thm-exit} and Proposition~\ref{lem-soft} for the full result).

\begin{theorem}\label{thm-main} For any $x\in W$ and any $r \in(0,\infty)$, as $t\to\infty$,
\begin{equation}
 \mathbb{P}_x\big(B_{[0,t]}\subset W \cap r(t)I^k\big) \sim h(x)
\begin{cases}K_0 r(t)^{-\frac k 2(k-1)} \e^{-t r(t)^{-2}\lambda^{\ssup {W\cap I^k}}}, &\mbox{if }1 \ll r(t) \ll \sqrt{t},\\
K_r t^{-\frac k 4(k-1)}, &\mbox{if }r(t)\sim r\sqrt{t},\\
K_{\infty} t^{-\frac k 4(k-1)}, &\mbox{if }\sqrt{t} \ll r(t).
\end{cases}
\end{equation}
Here $K_r\in(0,\infty)$ are constants for $r\in[0,\infty]$ such that 
\begin{equation}
\lim_{r\to \infty}K_r =K_{\infty}\qquad\mbox{and}\qquad
K_r\sim K_0 r^{-\frac k 2(k-1)} \e^{-r^{-2}\lambda^{\ssup {W\cap I^k}}}\quad\mbox{as }r\downarrow 0. 
\end{equation}
\end{theorem}

Interestingly, this shows that in the interpolating regime where $1 \ll r(t) \ll \sqrt{t}$, the polynomial decay term is already present; however, it does not come from the time parameter, but from the spatial parameter. It arises from the rescaling limit of the principal eigenfunction.

It is clear that the spectral decomposition method used in this paper is also able to describe the limiting conditional distribution of the endpoint of the Brownian motion given that the path stays in the truncated Weyl chamber for a long time; it is given in terms of the $L^1$-normalised principal eigenfunction:
$$
\mathbb{P}_x\big(B(t) \in \d y \,\big|\, B_{[0,t]}\subset W \cap I^k\big) \Longrightarrow \frac{f^{\ssup{W\cap I^k}}(y)}{\langle f^{\ssup{W\cap I^k}},\1\rangle}\,\d y,
$$
where the convergence is in the weak topology on  $W \cap I^k$. The second main question that we address is the description of these endpoints if the dimension $k$ grows to infinity, at times and in boxes that are allowed to grow unboundedly as a function of $k$, but do not have to. More precisely, we will give a large-deviation principle for the empirical measure of the endpoints of the $k$ single motions, properly rescaled, and identify the rate function explicitly with the help of some recent result by Eichelsbacher and Stolz. This in particular leads to a law of large numbers for this empirical measure in the spirit of the famous Wigner semi-circle law. However, the rate function and therefore the limiting probability measure have a different form, as the growing boundary of $r_k I$ is still felt in this limit.

More precisely, writing $B=B^{\ssup k}=(B_1,\dots,B_k)$, we consider the empirical measure of the properly transformed and rescaled end points of the $k$ Brownian motions, $B_1(t_k),\dots,B_k(t_k)$, 
\begin{equation}\label{mudef}
\mu^{\ssup k}_{r_k, t_k}= \frac 1k \sum_{i=1}^k \delta_{\sin(B_i(t_k)/r_k)},
\end{equation}
which is a random element of the set $\mathcal{M}_1([-1,1])$ of probability measures on $[-1,1]$.
A short version of our main result here, Theorem~\ref{thm-LDP}, reads as follows.

\begin{theorem}[Large-deviations principle]\label{thm-LDPmain}  
Suppose that the sequences $(r_k)_k$ and $(t_k)_k$ in $(0,\infty)$ fulfill $t_k\geq 16 r_k^2$. Then, as $k\to\infty$, uniformly in  $x\in W\cap r_k I^k$, the distribution of $\mu^{\ssup k}_{r_k, t_k}$ under $\P_x(\,\cdot\,|\, B_{[0,t_k]}^{\ssup k}\subset W\cap r_k I^k)$ satisfies a large-deviation principle on $\mathcal{M}_1([-1,1])$ with speed $k^2$ and rate function 
\begin{equation}\label{RF}
R(\mu)=\frac{1}{2}\int_{-1}^1\int_{-1}^1 \log|x-y|^{-1}\,\mu(\d x)\mu(\d y) - d,\qquad \mu\in\mathcal{M}_1([-1,1]),
\end{equation} 
where $d\in\R$ is such that $\inf_{\mu\in\mathcal{M}_1([-1,1])}R(\mu)=0$.
\end{theorem}


Explicitly, the statement of Theorem~\ref{thm-LDPmain} is that $R$ is a lower semicontinuous function and that, for any open set $F\subset \mathcal{M}_1([-1,1])$ and for any closed subset $G\subset \mathcal{M}_1([-1,1])$,
\begin{eqnarray*}
\liminf_{k\to\infty}\frac 1{k^2}\log \P_x(\mu^{\ssup k}_{r_k, t_k}\in F\,|\, B_{[0,t_k]}^{\ssup k}\subset W\cap r_k I^k)&\geq& -\inf_{\mu\in F} R(\mu),\\
\limsup_{k\to\infty}\frac 1{k^2}\log \P_x(\mu^{\ssup k}_{r_k, t_k}\in G\,|\, B_{[0,t_k]}^{\ssup k}\subset W\cap r_k I^k)&\leq& -\inf_{\mu\in G} R(\mu).
\end{eqnarray*}

Actually, a related large-deviations principle with the same rate function $R$ has recently been derived by Eichelsbacher and Stolz \cite{Eichelsb} for the empirical measure of the eigenvalues of a certain random matrix with explicit joint distribution of the components in terms of an orthogonal polynomial ensemble. Via the spectral decomposition method, we show that the joint distribution of $\sin(B^{\ssup k}(t_k)/r_k)$ is asymptotically sufficiently close to that ensemble. We find it remarkable that no divergence of the time $t_k$ nor of the radius $r_k$ is required; apparently no convergence to the invariant distribution is necessary.

From the principle in Theorem~\ref{thm-LDPmain}, a law of large numbers in the spirit of Wigner's semicircle theorem is derived as follows (see Cor.~\ref{Cor-LLN}).  Let the situation of Theorem~\ref{thm-LDPmain} be given. 

\begin{corollary}[Law of large numbers] As $k\to\infty$, uniformly in  $x\in W\cap r_k I^k$, the distribution of $\mu^{\ssup k}_{r_k, t_k}$ under $\P_x(\,\cdot\,|\, B_{[0,t_k]}^{\ssup k}\subset W\cap r_k I^k)$ converges weakly towards the arcsine distribution on $[-1,1]$.
\end{corollary}

The remainder of the paper is devoted to the proper formulation of the main results and their proofs. Actually, we do not treat the Weyl chamber $W_A$ only, but all the three Weyl chambers $W_Z = W_A, W_C, W_D$ given by
\begin{eqnarray*}
W_A & = & \{x = (x_1, \ldots, x_k) \in \mathbb{R}^k \colon x_1 < \ldots < x_k \}, \\
W_C & = & \{x = (x_1, \ldots, x_k) \in \mathbb{R}^k \colon 0 < x_1 <
\ldots < x_k \}, \\
W_D & = & \{x = (x_1, \ldots, x_k) \in \mathbb{R}^k \colon |x_1| < x_2 <
\ldots < x_k \}.
\end{eqnarray*}
In connection with Brownian motion, these chambers appeared first in a work by Grabiner \cite{Gra}. They are closely connected to so called alcoves of Weyl groups, which were investigated by Krattenthaler \cite{Krat}, Grabiner \cite{Grabiner} and Doumerc and Moriarty \cite{Mor}.

One can also consider the Brownian motion conditioned never to hit the boundary of $W\cap I^k$. Specialised to our situation, Pinsky \cite{Pinsky} showed that this process has generator $\frac 12\Delta + \frac{\nabla f^{\ssup{W\cap I^k}}}{f^{\ssup{W\cap I^k}}}\nabla$. This process is stationary, and its invariant distribution has $(f^{\ssup{W\cap I^k}})^2$ as Lebesgue density.

The paper is organized as follows: in the next section we set up the eigenfunction expansions that are essential for our purposes. In the subsequent section we use this machinery to prove the asymptotics for the different regimes and the soft transitions between them. In the final section we prove the large deviation principle and the law of large numbers.

\section{Eigenfunction Expansions}\label{sec-EigenfExp}

\noindent In this section, we give the details of the eigenvalue expansions for the Brownian motion before exiting any of the truncated Weyl chambers $W_Z\cap I^k$ for $Z=A,C,D$. In particular, we explicitly identify all the eigenvalues and eigenfunctions of one half times the negative Dirichlet Laplacian, $-\frac 12 \Delta$, in these three sets.

It is well-known that the non-exiting problem from an open bounded connected domain $U\subset \R^k$ is closely linked with the eigenvalues and eigenfunctions of the Dirichlet Laplacian in $U$. Let $\tau_U = \inf \{t>0\colon B(t) \notin U \}$ be the first exit time of the Brownian motion from the domain $U$. Then the events $\{B_{[0,t]}\subset U\} $ and $\{\tau_U>t\}$ are identical. The transition density of $B$ before exiting $U$ can be viewed as a symmetric positive definite operator on $L^2(\mathbb{R}^k)$ (see, for example, \cite{Stone}) and therefore admits the eigenfunction expansion uniformly in $ x,y\in U$ for $t>0$,
\begin{equation}\label{EigenfExp}
\mathbb{P}_x(B(t)\in \d y ; \tau_U >t)/\d y  = \sum_{l\in\N} \e^{-t\lambda_l^{\ssup U}}f_l^{\ssup U}(x)f_l^{\ssup U}(y),
\end{equation}
where $(\lambda_l^{\ssup U})_{l\in\N}$ is the spectrum of $-\frac 12\Delta$ with Dirichlet (i.e., zero) boundary condition in $U$, arranged in non-decreasing order, and $(f_l^{\ssup U})_{l\in\N}$ is a complete orthonormal system in $L^2(U)$ of corresponding eigenfunctions. The principal eigenvalue $\lambda^{\ssup U}=\lambda_1^{\ssup U}$ is simple and positive, and the corresponding eigenfunction $f_1^{\ssup U}=f^{\ssup U}$ is chosen strictly positive in $U$ (see for example \cite{Davies}). 

The key idea is to combine the expansion in \eqref{EigenfExp} for one-dimensional motions in $I$ with a Karlin-McGregor type formula to derive an expansion for the $k$-dimensional motion in the truncated Weyl chamber. This very natural method was already suggested by Hobson and Werner \cite{Hobson} who examined non-colliding Brownian motions on the circle. It avoids solving the heat equation with zero boundary condition in the truncated Weyl chamber, which would seem technically nasty.

We need the one-dimensional eigenfunction expansion. It is well-known that the spectrum and normalized eigenfunctions of $-\frac 12\Delta$ on $I=(-\frac\pi2,\frac\pi 2)$ with Dirichlet boundary condition are given by
\begin{equation}\label{spectrumd=1}
\lambda_l^{\ssup I}=\frac{l^2}{2},\qquad f_l^{\ssup I}=\sqrt{\frac 2\pi}\times
\begin{cases}
 \sin(lx),&\mbox{if }l\mbox{ is even,}\\
\cos(lx),&\mbox{if }l\mbox{ is odd.}
\end{cases}
\end{equation}
We could consider an abitrary symmetric interval instead of $I$, but we focus on $(-\frac\pi2, \frac\pi2)$ for convenience since then the formulas simplify. The eigenvalues and eigenfunctions on the interval $rI$ with $r>0$ are related by
\begin{equation}\label{Eigenfscaling}
\lambda_{l}^{\ssup {rI}}=r^{-2}\lambda_{l}^{\ssup I}, \qquad  f_{l}^{\ssup{rI}}(x)=r^{-1/2}f_l^{\ssup I}(x/r).
\end{equation}

The Karlin-McGregor-type formula for truncated Weyl chambers can be obtained from the original formula (see \cite{Karlin2}) by a small modification. For completeness, we give the proof. We abbreviate the density of the distribution of the one-dimensional Brownian motion before exiting the interval $I$ by 
\begin{equation}\label{ptIdef}
p_t^{\ssup I}(x,y)=\mathbb{P}_{x}(B_1(t)\in \d y;\tau_I>t)/\d y,\qquad x,y\in I.
\end{equation}

\begin{lemma}[Karlin-McGregor formula for a truncated Weyl chamber]\label{KarMcG}
For any $t>0$, and for any $x,y$ in $W_A$, $W_C$ and $W_D$, respectively,
\begin{eqnarray}
\mathbb{P}_x(B(t)\in \d y, \tau_{W_A \cap I^k} > t)/\d y & = &
\det\big[(p_t^{\ssup I}(x_i,y_j) )_{i,j=1,\dots,k}\big],\label{KMGA} \\
\mathbb{P}_x(B(t)\in \d y, \tau_{W_C \cap I^k} > t)/\d y & = &
\det\big[(p_t^{\ssup I}(x_i,y_j) - p_t^{\ssup I}(x_i,-y_j))_{i,j=1,\dots,k}\big],\label{KMGC} \\
\mathbb{P}_x(B(t)\in \d y, \tau_{W_D \cap I^k} > t)/\d y & = & 
\frac12\det\big[(p_t^{\ssup I}(x_i,y_j) - p_t^{\ssup I}(x_i,-y_j))_{i,j=1,\dots,k}\big] \label{KMGD} \\
 &  &\qquad + \frac 12\det\big[(p_t^{\ssup I}(x_i,y_j) + p_t^{\ssup I}(x_i,-y_j))_{i,j=1,\dots,k}\big].\nonumber
\end{eqnarray}
\end{lemma}

\begin{proof} We follow \cite[Sections 2 and 4]{Gra}, which gives the proof for $I^k$ replaced by $\R^k$. The same proof applies to our situation, since $I$ is symmetric around zero and is the same set in any of the $k$ dimensions.

The groups $Z = A, C, D$ corresponding to the Weyl chambers $W_Z$ consist of reflections $\R^k\to\R^k$, which are characterised by permutations of the components with sign changes of the components. The {\it symmetric group}, $A$, only permutes the components, $C$, the {\it hyperoctahedral group}, permutes the components with arbitrary sign changes and $D$, the {\it even hyperoctahedral group}, permutes the components with an even number of sign changes. If these reflections are understood as matrices, then $A$ is the set of all permutation matrices, $C$ is the set of all matrices that have precisely one real of modulus one in each row and each line, and zero otherwise, and $D$ is the set of all such matrices with an even number of $-1$. 

We prove the general formula
\begin{equation}\label{KMGformgeneral}
\mathbb{P}_x(B(t)\in \d y, \tau_{W_Z \cap I^k}>t) = 
\sum_{z\in Z} \sign(z) \mathbb{P}_x(B(t)\in \d z(y), \tau_{I^k}>t),
\end{equation}
where $z(y)=(\epsilon_1^{\ssup z} y_{\sigma_z(1)},\dots,\epsilon_k^{\ssup z} y_{\sigma_z(k)})\in\R^k$. Here $\epsilon_i^{\ssup z}\in\{-1,1\}$ denotes a possible sign change, $\sigma_z$ the permutation of the indices, and $\sign(z)=\sign(\sigma_z)\prod_i \epsilon_i^{\ssup z}$. Our assertions \eqref{KMGA}--\eqref{KMGD} can be deduced from \eqref{KMGformgeneral} by substituting the Weyl group $Z$. 

The idea is an application of the strong Markov property at time $\tau_{W_Z}$ and of an element of the Weyl group to the path $(B(\tau_{W_Z}+s))_{s\in[0,t-\tau_{W_Z}]}$. This uses that Brownian motion is a strong Markov process and that its increments are symmetric with respect to the Weyl groups, i.e., the distribution of $B(t_2)$ given $B(t_1)$ is, for $0\leq t_1<t_2$, the same as the distribution of $z(B(t_2))$ given $z(B(t_1))$. Hence, we can treat the difference of the two sides of \eqref{KMGformgeneral} as follows.
\begin{equation}\label{KMGproof1}
\begin{aligned}
& \mathbb{P}_x(B(t)\in \d y, \tau_{W_Z \cap I^k}>t) - 
\sum_{z\in Z} \sign(z) \mathbb{P}_x(B(t)\in \d y_z, \tau_{I^k}>t) = \\
& = \sum_{z\in Z} \sign(z) \left( \mathbb{P}_x(B(t)\in \d z(y), \tau_{W_Z \cap I^k}>t) - 
 \mathbb{P}_x(B(t)\in \d z(y), \tau_{I^k}>t) \right) \\
& =  \sum_{z\in Z} -\sign(z)  
 \mathbb{P}_x(B(t)\in \d z(y), \tau_{I^k}>t, \tau_{W_Z} \leq  t).
\end{aligned}
\end{equation}
Now we argue that the right hand side is equal to zero. Indeed, on $\{\tau_{W_Z} \leq  t\}$, we have $B(\tau_{W_Z})\in\partial W_Z$. We now construct some (random) $\sigma\in Z$ such that $\sign(\sigma)=-1$ and $\sigma(B(\tau_{W_Z}))=B(\tau_{W_Z})$. We distinguish some cases: If $B_i(\tau_{W_Z})=B_{i+1}(\tau_{W_Z})$ for some $i$, then pick $\sigma$ as the transposition of $i$ and $i+1$. If $Z=C$ and $B_1(\tau_{W_Z})=0$, then we pick $\sigma$ as the sign change for the first component. If $Z=D$ and $-B_1(\tau_{W_Z})=B_2(\tau_{W_Z})$, then pick $\sigma$ as the transposition of 1 and 2, together with two sign changes in the first two components. Note that the event $\{\tau_{I^k}>t\}$ remains unchanged when $(B(\tau_{W_Z}+s))_{s\in[0,t-\tau_{W_Z}]}$ is replaced by its image under $\sigma$, since $\sigma(I^k)=I^k$. Therefore, we have
$$
\begin{aligned}
\mbox{R.h.s.~of \eqref{KMGproof1}}& =  \sum_{z\in Z} -\sign(z)  
 \mathbb{P}_x(B(t)\in \d \sigma( z(y)), \tau_{I^k}>t, \tau_{W_Z} \leq t) \\
& = \sum_{z\in Z} \sign(\sigma \circ z)  
 \mathbb{P}_x(B(t)\in \d \sigma( z(y)), \tau_{I^k}>t, \tau_{W_Z} \leq t) \\
& = \sum_{\gamma\in Z} \sign(\gamma)  
 \mathbb{P}_x(B(t)\in \d \gamma(y), \tau_{I^k}>t, \tau_{W_Z} \leq t)\\
& =-\mbox{R.h.s.~of \eqref{KMGproof1}}.
\end{aligned}
$$
Hence, the term is equal to zero, and we are done.
\end{proof}

Now we use the eigenfunction expansion \eqref{EigenfExp} for $U=I$ in \eqref{KMGA}--\eqref{KMGD} to obtain the analogous expansions in the truncated Weyl chambers. We abbreviate, for a multi-index $l=(l_1,\dots,l_k)\in \N^k$ and $x=(x_1,\dots,x_k)\in I^k$,
\begin{equation}\label{eigenvZ}
\lambda_l^{\ssup Z}=\sum_{i=1}^k \lambda_{l_i}^{\ssup I}
\qquad\mbox{and}\qquad
f_l^{\ssup Z}(x)=\det\big[(f_{l_i}^{\ssup I}(x_j))_{i,j=1,\dots,k}\big]\times
\begin{cases}
 1,&\mbox{if }Z=A,\\
2^{k/2},&\mbox{if }Z=C,\\
2^{(k-1)/2},&\mbox{if }Z=D.
\end{cases}
\end{equation}
Furthermore, we need the three index sets
\begin{equation}\label{indexsets}
N_A=\N^k,\qquad N_C=(2\N)^k,\qquad    N_D= (2\N-1)^k \cup (2\N)^k.
\end{equation}

\begin{lemma}[Eigenvalue expansion in truncated Weyl chambers]\label{lem-EvExpWeyl}
The transition density of Brownian motion before exiting the truncated Weyl chamber $W_Z\cap I^k$ with $Z=A,C,D$ admits the following expansions, for any $t>0$, uniformly for $x,y\in W_Z\cap I^k$:
\begin{equation}\label{EvExpWeyl}
\mathbb{P}_x(B(t)\in \d y, \tau_{W_Z \cap I^k}>t)/\d y=\sum_{l\in W_A \cap N_Z} \e^{-t\lambda_l^{\ssup Z}} f_l^{\ssup Z}(x)f_l^{\ssup Z}(y).
\end{equation}
\end{lemma}

\begin{proof} Let us first prove the case A; we later explain the differences that occur in the two other cases, C and D.

We substitute the eigenvalue expansion \eqref{EigenfExp} for $p_t^{\ssup I}$ defined in \eqref{ptIdef} in \eqref{KMGA} to obtain 
\begin{equation}\label{EigenvExpProof1}
\begin{aligned}
\mathbb{P}_x(B(t)\in \d y, &\tau_{W_A \cap I^k} > t)/\d y
= \det\Big[\Big(\sum_{l=1}^{\infty} \e^{-t\lambda_{l}^{\ssup I}} f_{l}^{\ssup I}(x_i)f_{l}^{\ssup I}(y_j)\Big)_{i,j=1,\dots,k}\Big]\\
&= \sum_{l=(l_1,\dots,l_k)\in \N^k}\prod_{j=1}^k\e^{-t\lambda_{l_j}^{\ssup I}} \det\Big[\big(  f_{l_j}^{\ssup I}(x_i)f_{l_j}^{\ssup I}(y_j)\big)_{i,j=1,\dots,k}\Big],
\end{aligned}
\end{equation}
where we also used the multilinearity of the determinant in columns. Observe that the last determinant is identically zero if the $k$ indices $l_1,\dots,l_k$ are not pairwise distinct. Indeed, if $l_i=l_j$ for some $i\not=j$, then at least the $i$th and the $j$th row of the matrix are multiples of each other for all $x,y\in W_A\cap I^k$. Hence, the sum on $l\in\N^k$ may be reduced to the sum on $l\in W_A\cap\N^k$ with an additional sum on $\beta\in\Sym_k$, the set of all permutations of $1,\dots,k$, and $l$ is replaced by $l_\beta=(l_{\beta(1)},\dots,l_{\beta(k)})$. Using also the notation in \eqref{eigenvZ} for the eigenvalue, this gives
\begin{equation}\label{EigenvExpProof2}
\begin{aligned}
\mbox{R.h.s.~of \eqref{EigenvExpProof1}}&=\sum_{l=(l_1,\dots,l_k)\in W_A\cap \N^k}\e^{-t\lambda_{l}^{\ssup A}}\sum_{\beta\in \Sym_k}\det\Big[\Big(  f_{l_{\beta(j)}}^{\ssup I}(x_i)f_{l_{\beta(j)}}^{\ssup I}(y_j)\Big)_{i,j=1,\dots,k}\Big].
\end{aligned}
\end{equation}
Let us evaluate the sum on $\beta$. Using the substitutions $j = \tau^{-1}\circ\beta^{-1}(i)$ and $\tau^{-1}\circ\beta = \sigma$ for $\beta,\tau \in \Sym_k$, we compute 
$$
\begin{aligned}
\sum_{\beta\in \Sym_k}&\det\Big[\Big(  f_{l_{\beta(j)}}^{\ssup I}(x_i)f_{l_{\beta(j)}}^{\ssup I}(y_j)\Big)_{i,j=1,\dots,k}\Big]\\
&=\sum_{\beta, \tau} \sign(\tau) \prod_{j=1}^k \big[f_{l_{\beta\circ\tau(j)} }^{\ssup I}(x_j) f_{ l_{\beta\circ\tau(j)}}^{\ssup I}(y_{\tau(j)} )\big]\\
& = \sum_{\beta,\tau} \sign(\tau) \prod_{i=1}^k \big[f_{l_i}^{\ssup I}( x_{\tau^{-1}\circ\beta^{-1}(i)} )f_{l_i}^{\ssup I}( y_{\beta^{-1}(i)} )\big]\\
& = \sum_{\beta,\tau} \sign(\tau) \prod_{i=1}^k \big[f_{l_i}^{\ssup I}( x_{\tau^{-1}\circ\beta(i)} )f_{l_i}^{\ssup I}( y_{\beta(i)} )\big]\\
&= \sum_{\beta,\sigma} \sign(\beta)\sign(\sigma) \prod_{i=1}^k \big[f_{l_i}^{\ssup I}( x_{\sigma(i)} )f_{l_i}^{\ssup I}( y_{\beta(i)} )\big] \\
& = \Big(\sum_{\beta}  \sign(\beta) \prod_{i=1}^k f_{l_i}^{\ssup I}( y_{\beta(i)} )\Big)\Big( \sum_{\sigma}  \sign(\sigma) \prod_{j=1}^k f_{l_j}^{\ssup I}( x_{\sigma(j)} )\Big) \\
& = f_l^{\ssup A}(x)f_l^{\ssup A}(y) ,
\end{aligned}
$$
where we used the notation  in \eqref{eigenvZ} for the eigenfunction in the last step. Using this in \eqref{EigenvExpProof2}, we see that the proof of the lemma for $Z=A$ is complete.

Now we explain the differences to cases C and D. In the case C, inserting the eigenvalue expansion \eqref{EigenfExp} for $U=I$ in the formula \eqref{KMGC}, recalling \eqref{spectrumd=1} and using that the cosine is an even function and sine an odd one, we see that all cosine terms disappear and all sine terms appear twice, more precisely,
$$ 
\mathbb{P}_x(B(t)\in \d y, \tau_{W_C \cap I^k} > t)/\d y
= \det\Big[\Big(\sum_{l=1}^{\infty} 2 \e^{-t\lambda_{2l}^{\ssup I}} f_{2l}^{\ssup I}(x_i)f_{2l}^{\ssup I}(y_j)\Big)_{i,j=1,\dots,k}\Big].
$$
Hence, only even indices appear, and a factor of $2^k$ can be extracted from the determinant and is distributed to the two functions $f_{2l}^{\ssup C}(x)$ and $f_{2l}^{\ssup C}(y)$, see the second line in \eqref{eigenvZ}.

Case D is similar; from (\ref{KMGD}) we see that the first determinant is the same as in case C, and in the second only cosines remain:
\begin{eqnarray*}
& \mathbb{P}_x(B(t)\in \d y, \tau_{W_D \cap I^k} > t)/\d y  = \frac{1}{2} \det\Big[\Big(\sum_{l=1}^{\infty} 2 \e^{-t\lambda_{2l}^{\ssup I}} f_{2l}^{\ssup I}(x_i)f_{2l}^{\ssup I}(y_j)\Big)_{i,j=1,\dots,k}\Big] \\ 
& + \frac{1}{2} \det\Big[\Big(\sum_{l=1}^{\infty} 2 \e^{-t\lambda_{2l-1}^{\ssup I}} f_{2l-1}^{\ssup I}(x_i)f_{2l-1}^{\ssup I}(y_j)\Big)_{i,j=1,\dots,k}\Big]  .
\end{eqnarray*}
Now one easily sees how the prefactors $2^{k/2}$, $2^{(k-1)/2}$ and the index sets $N_C$, $N_D$ arise. 
\end{proof}

\begin{corollary}
For $Z=A,C,D$, the negative Dirichlet Laplacian $-\frac 12\Delta$ on $W_Z\cap I^k$ has spectrum $\{\lambda_l^{\ssup Z}\colon l\in W_A\cap N_Z\}$, where these eigenvalues are counted with multiplicity. Furthermore, $\{f_l^{\ssup Z}\colon l\in W_A\cap N_Z\}$ is a complete orthonormal system of corresponding eigenfunctions.
\end{corollary}

\begin{proof}
The functions $f_l^{\ssup Z}$ with $l\in W_A\cap N_Z$ are orthonormal on $L^2(W_Z\cap I^k)$ and they are eigenfunctions of $-\frac 12\Delta$ corresponding to the eigenvalues $\lambda_l^{\ssup Z}$, since the $f_l^{\ssup Z}$ are linear combinations of products of one-dimensional eigenfunctions which are orthonormalised on $I$, and the Laplacian is a linear operator. For the reader's convenience, we detail this. We concentrate on case A since the other cases follow in the same spirit.
First the eigenfunction property:
$$
\begin{aligned}
- \frac 12 \Delta f_l^{\ssup A}(x) & =  - \frac 12 \Delta \det\big[(f_{l_i}^{\ssup I}(x_j))_{i,j=1,\dots,k}\big] =  - \frac 12  \sum_{\sigma} \sign(\sigma) \Delta \prod_{i=1}^k f_{l_i}^{\ssup I}(x_{\sigma(i)}) \\
& = \sum_{\sigma} \sign(\sigma) \left(\sum_{i=1}^k \lambda_{l_i}^{\ssup I}\right)  \prod_{i=1}^k f_{l_i}^{\ssup I}(x_{\sigma(i)})
= \left(\sum_{i=1}^k \lambda_{l_i}^{\ssup I}\right) f_l^{\ssup A}(x)   \\
& =  \lambda_l^{\ssup A} f_l^{\ssup A}(x),
\end{aligned}
$$
where we also used \eqref{spectrumd=1} and \eqref{eigenvZ}.
The boundary condition is obviously satisfied because of the boundary condition of the onedimensional eigenfunctions and the determinantal structure. Now orthonormality for two multi-indices $l^1,l^2$:
$$
\begin{aligned}
\int_{W_A \cap I^k} f_{l^1}^{\ssup A}(x)f_{l^2}^{\ssup A}(x) dx 
& =  \frac{1}{k!} \int_{I^k} f_{l^1}^{\ssup A}(x)f_{l^2}^{\ssup A}(x)\, \d x   \\
&=\frac{1}{k!} \sum_{\alpha, \beta} \sign(\alpha\circ\beta) \int_{I^k} \prod_{i=1}^k f_{l_i^1}^{\ssup I}(x_{\alpha(i)}) f_{l_i^2}^{\ssup I}(x_{\beta(i)}) \,\d x\\
&=\frac{1}{k!} \sum_{\alpha, \beta} \sign(\alpha\circ\beta) \prod_{i=1}^k\Big\langle f_{l^1_i}^{\ssup I},f_{l^2_{\alpha\circ\beta^{-1}(i)}}^{\ssup I}\Big\rangle,
\end{aligned}
$$
where we wrote $\langle\cdot,\cdot\rangle$ for the standard inner product on $\R$. If $l^1\neq l^2$, then, for any $\alpha,\beta$, there is at least one $i$ such that $l^1_i\neq l^2_{\alpha\circ\beta^{-1}(i)}$, and hence the corresponding inner product is zero, since the $f_l^{\ssup I}$ form an orthonormal basis. If $l^1=l^2$, then for any $\alpha\neq \beta$, there is also at least such an $i$, such that the sum reduces to the sum on $\alpha=\beta$, which gives that the right-hand side is equal to one. This shows orthonormality.

These are in fact {\em all} eigenfunctions since otherwise there is a function $g\neq 0$ such that
$$
0=\sum_{l\in W_A \cap N_Z} \e^{-t\lambda_l^{\ssup Z}}  {\langle f_l^{\ssup Z}, g\rangle}^2 =\int \int g(y)g(x) \mathbb{P}_x(B(t)\in \d y, \tau_{W_Z \cap I^k}>t) \,\d x .
$$
But this contradicts the existence of an expansion of the transition density in terms of  a complete orthonormal system, recall \cite{Stone}.
\end{proof}

Note that, for $k\geq 3$, some of the eigenvalues $\lambda_l^{\ssup Z}$ coincide for different $l$, i.e., their multiplicity is larger than one. Examples of such eigenvalues can be constructed using Pythagorean number triples.

\begin{remark}
In particular the principal eigenvalues and eigenfunctions of $-\frac12\Delta$ in $W_Z\cap I^k$ with Dirichlet boundary condition are given by 
\begin{equation}\label{eigenvWeyl}
\lambda^{\ssup A}=\lambda_{\rm id}^{\ssup A} = \frac{1}{2}\sum_{i=1}^k i^2, \qquad \lambda^{\ssup C} = \lambda_{2\rm id}^{\ssup C}= 4\lambda^{\ssup A}, \qquad  
\lambda^{\ssup D} = \lambda_{2{\rm id}-1}^{\ssup D} = \frac{1}{2}\sum_{i=1}^k (2i-1)^2,
\end{equation}
and
\begin{equation}\label{eigenfident}
f^{\ssup A}= |f^{\ssup A}_{\rm id}|, \qquad f^{\ssup C}=2^{\frac{k}{2}} |f^{\ssup A}_{2{\rm id}}|, \qquad f^{\ssup D} =2^{\frac{k-1}{2}} |f^{\ssup A}_{2{\rm id}-1}|,
\end{equation}
where ${\rm id}=(1,2,3,\dots,k)$.
\end{remark}
Hence, $f^{\ssup Z}=f^{\ssup{W_Z \cap I^k}}$ in the notation of Section 1.
We are able to give explicit expressions for the principal eigenfunctions in terms of the {\it r\'eduites}. These are, by definition, positive harmonic functions for $-\frac 12\Delta$ that vanish on the boundary of the Weyl chambers. They are unique, up to positive multiples. They are given by
\begin{equation}\label{reduite}
h_A(x)=\det\big[(x_i^{j-1})_{i,j=1,\dots,k}\big], \qquad h_D(x)= h_A(x^2), \qquad h_C(x)=h_D(x)\prod_{i=1}^k x_i,
\end{equation}
where we wrote $x^2$ for the vector $(x_1^2,\dots,x_k^2)$. Note that $h=h_A$ is the classical Vandermonde determinant. The following identification clarifies the relation between the functions appearing in the asymptotics \eqref{Wasy} and \eqref{Wboxasy}. It also shows that it will be natural to consider the sine of the endpoints of the motions instead of the motions themselves, see \eqref{mudef}.

\begin{corollary}[Principal eigenfunctions]\label{cor-princEF}
\begin{eqnarray}
f^{\ssup A}(x) & = & \frac{2^{k^2/2}}{\pi^{k/2}} h_A(\sin(x)) \prod_{i=1}^k \cos(x_i),\label{fAident}  \\
f^{\ssup C}(x) & = & \frac{2^{k(k+1)}}{\pi^{k/2}} h_C(\sin(x)) \prod_{i=1}^k \cos(x_i), \label{fCident} \\
f^{\ssup D}(x) & = &  \frac{2^{(2k^2 - 1)/2}}{\pi^{k/2}} h_D(\sin(x)) \prod_{i=1}^k \cos(x_i)  .\label{fDident}
\end{eqnarray}
\end{corollary}

\begin{proof} Let us first consider the case A. Use \eqref{eigenfident} and \eqref{eigenvZ} (recall \eqref{spectrumd=1}) to see that
\begin{equation}
f^{\ssup A}(x)=\Big(\frac 2\pi\Big)^{k/2}\Big|\det\Big[ \big(\cos(i x_j)\1_{\{i\mbox{ odd}\}}+\sin(i x_j)\1_{\{i\mbox{ even}\}}\big)_{i,j=1,\dots,k}\Big]\Big|.
\end{equation}
Now use the well-known sine and cosine expansions for $i$ odd in the cosine and for $i$ even in the sine:
\begin{eqnarray}
\cos(ix) & = & \cos(x) \sum_{n=0}^{(i-1)/2} (-1)^n {i \binom 2n} (\sin^2(x))^n (1 - \sin^2(x))^{(i-1)/2-n},\label{cosexp}\\
\sin(ix)  &=&  \cos(x)\sin(x) \sum_{n=1}^{i/2} (-1)^{n+1} {i \binom 2n-1} (\sin^2(x))^{n-1} (1 - \sin^2(x))^{i/2-n}.\label{sinexp}
\end{eqnarray}
Note that the degrees of the monomials in the expansions all have the same parity. 
We extract the factors $\cos(x_j)$ row-wise from the determinants so that the terms remaining in the $i$-th row are polynomials $p_i$ in $\sin(x_j)$, i.e.,
$$
f^{\ssup A}(x)=\Big(\frac 2\pi\Big)^{k/2}\prod_{i=1}^k \cos(x_i) \Big|\det\Big[\big(p_i(\sin(x_j))\big)_{i,j=1,\dots,k}\Big]\Big|.
$$
Now observe that $p_i$ has degree precisely equal to $i-1$ with highest coefficient coming from a summation of the binomial coefficients over all  summands: For $i$ odd,
\begin{equation}\label{piodd}
p_i(y)=\sum_{n=0}^{(i-1)/2} (-1)^n {i \binom 2n} y^{2n}(1-y^2)^{(i-1)/2-n}=y^{i-1}2^{i-1}(-1)^{(i-1)/2}+O(y^{i-3}),
\end{equation}
and for $i$ even:
\begin{equation}\label{pieven}
p_i(y)=y\sum_{n=1}^{i/2} (-1)^{n+1} {i \binom 2n-1} y^{2n-2}(1-y^2)^{i/2-n}=y^{i-1} 2^{i-1} (-1)^{i/2-1}+O(y^{i-3}).
\end{equation}
Therefore, one can apply elementary row operations in such a way that in each entry of the determinant only the leading monomial is left. Afterwards, we can extract from the $i$-th row the prefactor $2^{i-1}$ (forget about the signs since we consider the modulus) and are left with
$$
f^{\ssup A}(x)=\Big(\frac 2\pi\Big)^{k/2} \det\Big[\big(\sin^{i-1}(x_j)\big)_{i,j=1,\dots,k}\Big]
\prod_{i=1}^k \big[\cos(x_i) 2^{i-1} \big].
$$
Now summarize the terms and recall \eqref{reduite} to see that \eqref{fAident} is true.

Now we come to cases C and D. Plugging in the onedimensional eigenfunctions yields 
\begin{eqnarray*}
f^{\ssup C}(x) & = & \Big(\frac 2\pi\Big)^{k/2} 2^{k/2} \Big|\det\Big[ \big(\sin(2i x_j)\big)_{i,j=1,\dots,k}\Big]\Big| \\
f^{\ssup D}(x) & = & \Big(\frac 2\pi\Big)^{k/2} 2^{(k-1)/2}  
\Big|\det\Big[ \big(\cos((2i-1)x_j)\big)_{i,j=1,\dots,k}\Big]\Big|. 
\end{eqnarray*}
Using expansions (\ref{cosexp}) and (\ref{sinexp}) we obtain
\begin{eqnarray*}
f^{\ssup C}(x) & = & \frac{2^k}{\pi^{k/2}}  \Big|\det\Big[\big(p_{2i}(\sin(x_j))\big)_{i,j=1,\dots,k}\Big]\Big| \prod_{i=1}^k \cos(x_i)\\
f^{\ssup D}(x) & = & \frac{2^{k-1/2}}{\pi^{k/2}}  \Big|\det\Big[\big(p_{2i-1}(\sin(x_j))\big)_{i,j=1,\dots,k}\Big]\Big|\prod_{i=1}^k \cos(x_i). 
\end{eqnarray*}
For C and D the degrees of the polynomials in $\sin(x)$ increase by two with each row, so that we get the degrees from $1$ to $2k-1$ for $C$ and from $0$ to $2k-2$ for $D$. One can perform exactly the same row operations since all occuring monomials of the polynomials have the same parity in their degrees. But now we actually get $h_A$ in sine squares together with a product of sines in case $C$. Hence we arrive at (\ref{fCident}) and (\ref{fDident}) (recall (\ref{reduite})). 
\end{proof}

\section{Exit regimes}\label{LDPs}

\noindent Now we use our results on the eigenvalue expansions from Section~\ref{sec-EigenfExp} to identify the asymptotics of the non-exit probabilities in growing truncated Weyl chambers. For this we prove a technical lemma. Note that we abbreviate $\langle f^{\ssup Z},\1\rangle$ by $\int f^{\ssup Z}$. Abbreviate
\begin{equation}\label{growthfct}
\gamma(t):= -\ln\big(1- \e^{-(\frac t2-7)} \big) - (\sfrac t2-7),\qquad t>14.
\end{equation}

\begin{lemma}\label{lemunibou}
Fix $Z \in \{A,C,D\}$. Then, for any $t,r\in(0,\infty)$ with $t/r^2>14$ and for any $x,y \in W_Z\cap rI^k$,
\begin{equation}\label{isolateZ}
\begin{aligned}
\mathbb{P}_x&(B(t)\in \d y, \tau_{W_Z \cap rI^k} > t)/\d y \\
&= \e^{-tr^{-2}\lambda^{\ssup Z}} r^{-k} f^{\ssup Z}(x/r)f^{\ssup Z}(y/r) ( 1+\eps_{tr^{-2}}^{\ssup Z}(x/r,y/r)), 
\end{aligned}
\end{equation}
and
\begin{equation}\label{isolateintZ}
\mathbb{P}_x( \tau_{W_Z \cap rI^k} > t) =  \e^{-t r^{-2}\lambda^{\ssup Z}} f^{\ssup Z}(x/r) \int f^{\ssup Z}\, (1 + \widetilde{\eps}_{tr^{-2}}^{\ssup Z}(x/r)),
\end{equation}
where the error terms satisfy
\begin{equation}\label{uniferbou}
\sup_{x,y\in W_Z\cap I^k}|\epsilon_t^{\ssup Z}(x,y)| \leq \e^{k \gamma(t)},\qquad \sup_{x\in W_Z\cap I^k}|\widetilde{\epsilon}_t^{\ssup Z}(x)| \leq \e^{k \gamma(t)}.
\end{equation}
\end{lemma}

\begin{proof}
We detail the proof for $Z=A$ only and explain the differences to the other two types later. Use \eqref{EvExpWeyl}, \eqref{Eigenfscaling} and \eqref{eigenvWeyl} and isolate the first term in the expansion to get
\begin{equation}\label{isolate}
\begin{aligned}
\mathbb{P}_x&(B(t)\in \d y, \tau_{W_A \cap r I^k} > t)/\d y \\
&= \sum_{l=1}^{\infty} \e^{-t r^{-2}\lambda_{l}^{\ssup A}} r^{-k} f_{l}^{\ssup A}(x/r)f_{l}^{\ssup A}(y/r)  \\
&= \e^{-tr^{-2}\lambda^{\ssup A}} r^{-k} f^{\ssup A}(x/r)f^{\ssup A}(y/r) ( 1+\eps_{tr^{-2}}^{\ssup A}(x/r,y/r)), 
\end{aligned}
\end{equation}
where
\begin{equation}
\eps_t^{\ssup A}(x,y)= \sum_{l=(l_1, \dots, l_k)\in W_A \cap \mathbb{N}^k \setminus \{{\rm id}\}}  \e^{-\frac{t}{2} \sum_{i=1}^k (l_i^2 - i^2)} \frac{f_{l}^{\ssup A}(x)f_{l}^{\ssup A}(y)}{f^{\ssup A}(x)f^{\ssup A}(y)}.
\end{equation}
We first claim that
\begin{equation}\label{unifbou}
\sup_{x\in W_A\cap I^k}\Big|\frac{f_{l}^{\ssup A}(x)}{f^{\ssup A}(x)}\Big| \leq 2^{-k(k-1)/2}  \frac{h_A(\tilde{l})}{h_A(\rm{id})}\Big(\prod_{i\colon l_i>i} [2^{3l_i/2}l_i]\Big) \Big(\prod_{i\colon l_i=i} 2^{l_i} \Big),
\end{equation}
where $\tilde{l} \in W_A \cap \mathbb{N}^k \setminus \{{\rm id}\}$, maximizes $h_A$ subject to $\tilde{l} \leq l$; we understand the inequality componentwise. Its derivation will now be explained in detail.

As in the proof of Corollary~\ref{cor-princEF}, we see that, for any $l\in\N^k$,
\begin{equation}\label{fAlident}
f_l^{\ssup A}(x)=\Big(\frac 2\pi\Big)^{k/2} \det\Big[\big(p_{l_i}(\sin(x_j))\big)_{i,j=1,\dots,k}\Big]\prod_{i=1}^k \cos(x_i),
\end{equation}
where the polynomials $p_i$ are given in (\ref{piodd}) and (\ref{pieven}). The degree of $p_{l_i}$ is $l_i -1$, and the coefficients of all lower monomials with parity of degree different from the one of $l_i-1$ are zero. 

Now we evaluate the determinant. As in the proof of Corollary~\ref{cor-princEF}, we carry out suitable row operations to cancel in the polynomial of row $i$ every monomial of order $<i-1$. But now, to achieve this, we first need to suitably permute all rows $i$ satisfying $l_i>i$. Let us call the arising vector $l'$. Hence, there are polynomials 
$$
\widetilde p_{i,l_i '}(w)=\sum_{n=i}^{l_i '}w^{n-1} b_{n,i,l_i'},\qquad w\in\R,
$$
with suitable coefficients $b_{n,i,l_i'}$ such that 
$$
\Big|\det\Big[\big(p_{l_i}(\sin(x_j))\big)_{i,j=1,\dots,k}\Big]\Big|=\Big|\det\Big[\big(\widetilde p_{i,l_i '}(\sin(x_j))\big)_{i,j=1,\dots,k}\Big]\Big|.
$$
These coefficients satisfy $|b_{n,i,l_i '}|\leq 2^{3l_i '/2}$ if $l_i '>i$ and $|b_{n,i,l_i '}|\leq 2^{l_i '}$ if $l_i '=i$. This is explained as follows: if $l_i '=i$, then $2^{l_i '}$ bounds the sum of the binomial coefficients for each monomial in (\ref{piodd}) and (\ref{pieven}); if $l_i '>i$, then we need the additional power of $l_i'/2$ due to the binomial coefficients which arise by expansion of the power of $(1-y^2)$ in (\ref{piodd}) and (\ref{pieven}).

Using the multilinearity of the determinant, we obtain
$$
\det\Big[\big(\widetilde p_{i,l_i '}(\sin(x_j))\big)_{i,j=1,\dots,k}\Big]
=\sum_{\heap{i\leq n_i\leq l_i '}{i=1,\dots,k}}a_{n}(\sin(x))\prod_{i=1}^k b_{n_i,i,l_i '} ,
$$
where $a_{(n_1,\dots,n_k)}(w)=\det[(w_j^{n_i-1})_{i,j=1,\dots,k}]$ for $w=(w_1,\dots,w_k)$.
Now we introduce the {\it Schur polynomials},
$$
s_{d}(w)=  \frac{a_{d+\rm{id}}(w)}{h_A(w)},\qquad w\in\R^k,
$$
where $d= (d_1, \dots, d_k)\in\N_0^k$ satisfies $d_1\leq \dots \leq d_k$, see e.~g.~\cite{Fulton}. To be able to employ these polynomials, we associate to each $n\in\N_0^k$ its increasingly ordered version $\overrightarrow{n}$. Then $a_{ \overrightarrow{n}}$ differs at most by a sign change from $a_n$. Note that if $n_i = n_j$ for at least two indices $i$ and $j$, then $a_n$ and hence $a_{ \overrightarrow{n}}$ is identically zero. Using \eqref{fAlident} for $f_{l}^{\ssup A}$ and (\ref{fAident}) for $f^{\ssup A}$, we see that
$$
\begin{aligned}
\Big|\frac{f_{l}^{\ssup A}(x)}{f^{\ssup A}(x)}\Big|
&=\left|\frac{ \det\big[\big(p_{l_i}(\sin(x_j))\big)_{i,j=1,\dots,k}\big]}{ 2^{k(k-1)/2} h_A(\sin(x))}\right|\\
&\leq 2^{-k(k-1)/2}\sum_{\heap{i\leq n_i\leq l_i '}{i=1,\dots,k; n_i\neq n_j}}|s_{\overrightarrow{n}-{\rm{id}}}(\sin(x))|\prod_{i=1}^k |b_{n_i,i,l_i'}|.
\end{aligned}
$$
Now we estimate the modulus of the right-hand side. Note that $s_{\overrightarrow{n}-{\rm{id}}}(\sin(x))$ is a multipolynomial in $\sin(x_1),\dots,\sin(x_k)$ with positive coefficients and that all these arguments are in $[-1,1]$. Therefore, 
$$
|s_{\overrightarrow{n}-{\rm{id}}}(\sin(x))|\leq s_{\overrightarrow{n}-{\rm{id}}}(\1)= \frac{|h_A(n)|}{h_A({\rm{id}})}\leq \frac{h_A(\tilde{l})}{h_A({\rm{id}})},
$$
see \cite{Fulton} or \cite[proof of Lemma~2.3]{J02}. Hence, we have
\begin{eqnarray*}
\sup_{x\in W_A\cap I^k}\Big|\frac{f_{l}^{\ssup A}(x)}{f^{\ssup A}(x)}\Big|
\leq 2^{-k(k-1)/2}\frac{h_A(\tilde{l})}{h_A({\rm{id}})}
\Big(\prod_{i\colon l_i>i} 2^{3l_i/2}l_i\Big)\Big( \prod_{i\colon l_i=i} 2^{l_i}\Big).
\end{eqnarray*}

This proves (\ref{unifbou}) which we can now plug in the error term $\epsilon_t^{\ssup A}(x,y)$:
$$
\begin{aligned}
\sup_{x, y\in W_A\cap I^k}|\eps_t^{\ssup A}(x,y)|
& \leq \sum_{l\in W_A \cap \mathbb{N}^k \setminus \{{\rm id}\}}   
\e^{-\frac t2\sum_{i=1}^k (l_i^2 - i^2)} \left|\frac{f_{l}^{\ssup A}(x)f_{l}^{\ssup A}(y)}{f^{\ssup A}(x)f^{\ssup A}(y)}\right| \\
& \leq \sum_{l\in W_A \cap \mathbb{N}^k \setminus \{{\rm id}\}} 2^{-k(k-1)} \e^{-\frac t2\sum_{i\colon l_i>i} (l_i - i)(l_i + i)}  \\
& \times \left( \frac{h_A(\tilde{l})}{h_A({\rm{id}})}
\Big(\prod_{i\colon l_i>i} 2^{3l_i/2}l_i\Big)\Big( \prod_{i\colon l_i=i} 2^{l_i}\Big) \right)^2.
\end{aligned}
$$
With help of the elementary estimate
$$
\begin{aligned}
  \ln\left(\frac{h_A(\tilde{l})}{h_A(\rm{id})}\right) &\leq 
\sum_{i,j\colon j<i<\tilde l_i} \ln\frac{\tilde{l}_i  - j}{i-j} 
 = \sum_{i,j\colon j<i<\tilde l_i} \ln\left( 1 + \frac{\tilde{l}_i - i}{ i - j}\right)  \\
&  \leq  \sum_{i,j\colon j<i<\tilde l_i}  \ln(2(\tilde{l}_i - i) ) 
\leq \sum_{i\colon \tilde l_i>i}  (i-1) 2({l}_i - i)
\leq \sum_{i\colon l_i>i} ({l}_i + i)({l}_i - i)
\end{aligned}
$$
and using that  $2^{-k(k-1)}( \prod_{i\colon l_i=i} 2^{l_i})^2\leq 1$, we can proceed by
$$
\begin{aligned}
 \sup_{x, y\in W_A\cap I^k}&|\eps_t^{\ssup A}(x,y)|  \\
&\leq \sum_{l\in W_A \cap \mathbb{N}^k \setminus \{{\rm id}\}} \exp\left( 2\sum_{i\colon l_i>i} \big[({l}_i + i)({l}_i - i)+ l_i \sfrac 32\ln2  +  \ln(l_i) \big]  \right) \\
& \qquad\times \exp\left( - \frac t2\sum_{i\colon l_i>i} (l_i - i)(l_i + i) \right) \\
& \leq \sum_{l\in W_A \cap \mathbb{N}^k \setminus \{{\rm id}\}} \exp\left( - \left(\frac t2 -7\right)\sum_{i\colon l_i>i} (l_i - i)(l_i + i) \right),
\end{aligned}
$$
where we also estimated $l_i \sfrac 32\ln2  +  \ln(l_i) \leq \frac 52 ({l}_i + i)({l}_i - i)$.
Define $c_1(t):= \frac t2 -7$ and $c_2(t) := \frac{1}{1 - \e^{-c_1(t)}}$. Then under the assumption $t>14$, we use in the sum on $l$ that $l_i\geq i$ for $i=1,\dots, k-1$ and $l_k\geq k+1$ and compare to the geometric series, to obtain:
$$
\begin{aligned}
 \sup_{x, y\in W_A\cap I^k}|\eps_t^{\ssup A}(x,y)| &\leq \sum_{l\in W_A \cap \mathbb{N}^k \setminus \{{\rm id}\}}  \e^{-c_1(t)(l_1^2 - 1^2 + \cdots + l_k^2 - k^2)} \\ 
& = \sum_{l\in W_A \cap \mathbb{N}^k \setminus \{{\rm id}\}} \left(\e^{-c_1(t)}\right)^{l_1^2 - 1}\prod_{i=2}^k \e^{-c_1(t)(l_i^2 - i^2)}\\
& \leq \frac{1}{1 - \e^{-c_1(t)}} \sum_{(l_2, \dots, l_k)\in W_A \cap (\mathbb{N}+1)^{k-1} \setminus \{(2,\dots,k)\}} \prod_{i=2}^k
\e^{-c_1(t)(l_i^2 - i^2)} \\
& \leq (c_2(t))^{k-1} \sum_{l= k+1}^\infty \e^{-c_1(t)(l^2 - k^2)} 
= (c_2(t))^{k-1} \sum_{n= 1}^\infty \e^{-c_1(t)(2nk + n^2)} \\
& \leq (c_2(t))^{k-1} \e^{-kc_1(t)} \sum_{n= 1}^\infty \left(\e^{-c_1(t)}\right)^{(2n-1)k}\leq (c_2(t))^k \e^{-kc_1(t)} \\
&  = \e^{k \gamma(t)} ,
\end{aligned}
$$
where we recall the definition of $\gamma(t)$ from \eqref{growthfct}.
This proves the first bound in \eqref{uniferbou} for the error term in (\ref{isolateZ}) and therefore finishes the proof of \eqref{isolateZ} for the case A.

If we integrate $\mathbb{P}_x(B(t)\in \d y, \tau_{W_A \cap r I^k} > t)$ over $y$, we obtain
\begin{equation*}
\mathbb{P}_x( \tau_{W_A \cap r I^k} > t) = \sum_{l=1}^{\infty} \e^{-t r^{-2}\lambda_{l}^{\ssup A}} f_{l}^{\ssup A}(x/r) \int f_{l}^{\ssup A}.
\end{equation*}
Now one can isolate the first summand as in (\ref{isolate}) and carry out exactly the same procedure as above with the only difference that $f_l^{\ssup A}(y)$ is replaced by $\int f_l^{\ssup A}$. This yields \eqref{isolateintZ} with an error term $\tilde\eps$ satisfying the second bound in \eqref{uniferbou}. Hence, the proof of the lemma for Z $=$ A is finished.

For C and D we can use the same procedure with the only differences that some $l\in W_A \cap \mathbb{N}^k \setminus \{{\rm id}\}$ do not appear in the expansions and we now have to divide by Vandermonde determinants in sine squares together with a product of sines in case C. But this leads to the same bound since all components of the occuring $l$ are guaranteed to have the same parity. Hence the lemma is proved. 
\end{proof}

With the help of this lemma we can now formulate and prove our first main theorem.

\begin{theorem}[Non-exit from growing truncated Weyl chambers]\label{thm-exit} Fix $Z\in\{A,C,D\}$. Then, 
for any function $r \colon(1,\infty) \rightarrow (0,\infty)$, as $t$ goes to infinity, for $x\in W_Z \cap r(t)I^k$ and $r\in(0,\infty)$,
\begin{equation}
 \mathbb{P}_x\big(\tau_{W_Z \cap r(t)I^k}>t\big) \sim
\begin{cases}
\e^{-t r^{-2}\lambda^{\ssup Z}}f^{\ssup Z}(\frac{x}{r})\int f^{\ssup Z},&\mbox{if }r(t)\equiv r,\\
K_0^{\ssup Z} r(t)^{-\alpha_Z} h_Z(x) \e^{-t r(t)^{-2}\lambda^{\ssup Z}}, &\mbox{if }1 \ll r(t) \ll \sqrt{t},\\
K_r^{\ssup Z} h_Z(x) t^{-\alpha_Z/2}, &\mbox{if }r(t)\sim r\sqrt{t},\\
K_{\infty}^{\ssup Z} h_Z(x) t^{-\alpha_Z/2}, &\mbox{if }\sqrt{t} \ll r(t).
\end{cases}
\end{equation}
The convergence is uniform for $x\in W_Z \cap r(t)I^k$, without further restriction in the first case, with the restriction $|x| \leq \theta_t r(t)$ in the two middle cases and with the restriction $|x| \leq \theta_t \sqrt{t}$ in the last case, for any $0<\theta_t\to 0$ as $t\to\infty$. In the third line, $K_r^{\ssup Z} := \mathbb{P}_0(\tau_{rI^k}>1|\tau_{W_Z}>1)K_{\infty}^{\ssup Z}$. The other parameters are given as follows.
\begin{equation}
\alpha_A=\frac k2(k-1),\qquad \alpha_C=k^2,\qquad\alpha_D=k(k-1),
\end{equation}
and
\begin{equation}
\begin{array}{rclrcl}
K_0^{\ssup A}&=& \frac{2^{k^2/2}}{\pi^{k/2}}\int f^{\ssup A}, & K_{\infty}^{\ssup  A} & = & \frac{2^k \prod\limits_{i=1}^k \Gamma(i/2 + 1)}{\pi^{k/2} k! \prod\limits_{i<j}(j-i) }  \\
K_0^{\ssup C}&=& \frac{2^{k(k+1)}}{\pi^{k/2}}\int f^{\ssup C}, & K_{\infty}^{\ssup C}  & =&  \frac{2^{3k^2/2} \prod\limits_{i=1}^k \Gamma(i/2 + 1)\Gamma((i+1)/2)}{\pi^k k! \prod\limits_{i<j}[(2j-1)^2 - (2i-1)^2] \prod\limits_{i=1}^k (2k+1-2i) }  \\
K_0^{\ssup D}&= &\frac{2^{(2k^2 - 1)/2}}{\pi^{k/2}}\int f^{\ssup D}, & K_{\infty}^{\ssup D}  & =&  \frac{2^{(3k^2-3k+2)/2} \prod\limits_{i=1}^k \Gamma(i/2 + 1)\Gamma(i/2)}{\pi^k k! \prod\limits_{i<j}[(2j-1)^2 - (2i-1)^2] }.
 \end{array}
\end{equation}
\end{theorem}

\noindent{\it Remark.} The conditional probability appearing in the definition of $K_r^{\ssup Z}$ is to be interpreted as
\begin{equation}\label{condProb}
\mathbb{P}_0(\tau_{rI^k}>1|\tau_{W_Z}>1)=\lim_{x\rightarrow 0, x\in W_Z} \frac{\mathbb{P}_x(\tau_{rI^k}> 1, \tau_{W_Z}>1)}{\mathbb{P}_x(\tau_{W_Z}>1)},
\end{equation}
see \cite[Thm.~2.2]{KatoTaneV}.

\begin{proof} 
The assertions about the asymptotics of the non-exit probabilities in the first two regimes follow from (\ref{isolateintZ}) and (\ref{uniferbou}) of Lemma~\ref{lemunibou} since by the choices of $r(t)$ we have $\gamma(\frac{t}{r(t)^2})\rightarrow -\infty$ and furthermore $f^{\ssup Z}(x/r(t))\sim    K_0^{\ssup Z} r(t)^{-\alpha_Z}h_Z(x)/\int f^{\ssup Z} $ in the second regime.

Now we come to the proof of the last two regimes, for any $Z\in\{A,C,D\}$. In the third regime, where $r(t)/\sqrt t\to r$, we use Brownian scaling to see that
\begin{align*}
& \mathbb{P}_x(\tau_{W_Z \cap r(t)I^k} >t) = \mathbb{P}_{\frac{x}{\sqrt{t}}}\left(\tau_{rI^k}>1\,\big|\,\tau_{W_Z}>1\right) \mathbb{P}_x(\tau_{W_Z}>t).
\end{align*}
The asymptotics $\mathbb{P}_x(\tau_{W_Z}>t)\sim K_{\infty}^{\ssup Z} h_Z(x) t^{-\alpha_Z/2}$ are well-known due to \cite{Gra}. 
This is where the restriction $|x| \leq \theta_t \sqrt{t}$, with any $0<\theta_t\to 0$ as $t\to\infty$, is needed. In order to see that the first term on the right-hand side converges towards $K_r^{\ssup Z}=\mathbb{P}_{0}(\tau_{rI^k}>1\,|\,\tau_{W_Z}>1)$, we use \cite{KatoTaneV} that $(B_s)_{s\in[0,1]}$, conditional given $\{\tau_{W_Z}>1\}$, is a temporarily inhomogeneous diffusion process for which zero is an entrance boundary. In particular, we have $\lim_{y\to 0,y\in W_Z}\mathbb{P}_y(\tau_{rI^k}>1\,|\,\tau_{W_Z}>1)=\mathbb{P}_0(\tau_{rI^k}>1\,|\,\tau_{W_Z}>1)$, i.e., the proof in the third regime is done.

In the fourth regime, where $r(t)\gg \sqrt t$, we proceed similarly:
\begin{align*}
& \mathbb{P}_x(\tau_{W_Z \cap r(t)I^k} >t) = \mathbb{P}_{\frac{x}{\sqrt{t}}}\left(\tau_{r(t)t^{-1/2}I^k}>1\,\Big|\,\tau_{W_Z}>1\right)\mathbb{P}_x(\tau_{W_Z}>t).
\end{align*}
While the last term is handled in the same way as in the third regime, the first term is easily seen to converge to one. Indeed, it is not larger than one, and it is, for any fixed $r>0$ and  for any sufficiently large $t$, not smaller than  $\mathbb{P}_{\frac{x}{\sqrt{t}}}(\tau_{rI^k}>1\,|\,\tau_{W_Z}>1)$. Now carry out the limit as $t \to\infty$ using the above argument, and afterwards the limit as $r\uparrow \infty$.
\end{proof}

Furthermore, there is even a smooth transition between these regimes.

\begin{proposition}[Soft transition]\label{lem-soft} For $Z\in\{A,C,D\}$,
$$
\lim_{r\to\infty}K_r^{\ssup Z}= K_{\infty}^{\ssup Z},\qquad \mbox{and}\qquad
K_r^{\ssup Z}\sim K_0^{\ssup Z} \e^{-r^{-2}\lambda^{\ssup Z} } r^{-\alpha_Z} \quad\mbox{as } r\rightarrow 0.
$$
\end{proposition}

\begin{proof}
The first statement is obvious. For proving the second, we use \eqref{condProb} and substitute, in the denominator, the asymptotics $\mathbb{P}_x(\tau_{W_Z}>1)=K_{\infty}^{\ssup Z} h_Z(x)(1+o_x(1))$ as $x\rightarrow 0, x\in W_Z$, which easily follows via Brownian scaling from \cite{Gra}. Note that we can interchange the limits $x\rightarrow 0$ and $r\downarrow 0$ because of uniform convergence which follows from Lemma \ref{lemunibou}, see \eqref{isolateintZ}, since $\lim_{r\downarrow 0}\gamma(r^{-2})=-\infty$, see \eqref{growthfct}. This gives that 
$$
\begin{aligned}
K_r^{\ssup Z} &= \lim_{x\rightarrow 0, x\in W_Z} \frac{\mathbb{P}_x(\tau_{W_Z\cap rI^k}> 1)}{\mathbb{P}_x(\tau_{W_Z}>1)} K_{\infty}^{\ssup Z}  \\
& \sim \lim_{x\rightarrow 0, x\in W_Z} \frac{\e^{-r^{-2}\lambda^{\ssup Z}} f^{\ssup Z}(x/r) \int f^{\ssup Z}}{K_{\infty}^{\ssup Z} h_Z(x)(1+o_x(1))} K_{\infty}^{\ssup Z} \\
&= K_0^{\ssup Z} \e^{-r^{-2}\lambda^{\ssup Z}} r^{-\alpha_Z}.
\end{aligned}
$$
\end{proof}

\section{Large-deviation principle for diverging dimension}\label{sec-LDP}

\noindent Now we consider limits as the dimension $k$ diverges. Therefore, we now write $B^{\ssup k}=(B_1,\dots,B_k)$ for the $k$-dimensional Brownian motion.

By $\mathcal{M}_1([a,b])$ we denote the set of probability measures on $[a,b]$, with $a,b \in \mathbb{R}, a<b$. Recall that $\mu^{\ssup k}_{r_k,t_k}$ denotes the empirical measure of the vector $\sin(B^{\ssup k}(t_k)/r_k)$, see \eqref{mudef}. With the help of Lemma~\ref{lemunibou}, we can also prove large-deviation principles.

\begin{theorem}[LDP for diverging dimension]\label{thm-LDP} Assume that $Z\in\{A,C\}$. Let $(r_k)_{k\in\N}$ and $(t_k)_{k\in\N}$ be sequences in $(0,\infty)$ satisfying $t_k\geq 16 r_k^2$. Then, as $k\to\infty$, the conditional distribution of $\mu^{\ssup k}_{r_k,t_k}$ under $\P_x( \cdot\,|\, B^{\ssup k}_{[0,t_k]}\subset W_Z\cap r_k I^k)$ satisfies, uniformly in $x\in W_Z \cap r_k I^k$, a large deviation principle on $\mathcal{M}_1([-1,1])$ in the case $Z=A$ and on $\mathcal{M}_1([0,1])$ in the case $Z=C$ with respect to the weak topology with speed $k^2$ and good rate function 
\begin{eqnarray}
R_A(\mu)&=&\frac{1}{2}\int_{-1}^1\int_{-1}^1 \log|x-y|^{-1}\,\mu(\d x)\mu(\d y) - d_A,\label{ratefunction1}\\
R_C(\mu)&=&\frac{1}{2}\int_0^1\int_0^1 \log|x^2-y^2|^{-1}\,\mu(\d x)\mu(\d y) -\int_0^1 \log x \,\mu(\d x) - d_C,\label{ratefunction2}
\end{eqnarray} 
where $d_Z \in\R$ is such that $\inf R_Z=0$. 
\end{theorem}

It follows from the theory of logarithmic potentials with external fields, see \cite{Saff} for example, that $d_Z$ is finite. We also have $d_Z= \lim_{k\rightarrow\infty} \frac{1}{k^2} \log \int_{W_Z\cap (2I/\pi)^k} h_Z(x)\,\d x $.

Our proof of Theorem~\ref{thm-LDP} relies on a related principle for an orthogonal polynomial ensemble, proved by Eichelsbacher and Stolz \cite{Eichelsb}. However, the case Z $=$ D cannot be treated by them, due to the appearance of a square in the density of that ensemble, which leads to some ambiguity in the interpretation of the squareroot.

\begin{proof} We first claim that, as $k\to\infty$,
\begin{equation}\label{claim}
\mathbb{P}_x\Big(\sin\Big(\frac{B^{\ssup k}(t_k)}{r_k}\Big)\in \d y\,\Big|\, \tau_{W_Z\cap r_k I^k} >t_k\Big)/\d y \sim  
\frac{h_Z(y)}{\int_{W_Z\cap (2I/\pi)^k} h_Z(w)\,\d w},
\end{equation}
uniformly in $x\in W_Z \cap r_k I^k$ and $y\in W_Z \cap (2I/\pi)^k$. Indeed, if we apply the transformation $x\mapsto \sin(x/r_k)$ to $B^{\ssup k}(t_k)$ in (\ref{isolateZ})  of Lemma \ref{lemunibou}, we obtain, as $k\to\infty$,
$$
\begin{aligned}
&\mathbb{P}_x\Big(\sin\Big(\frac{B^{\ssup k}(t_k)}{r_k}\Big)\in \d y\, , \, \tau_{W_Z\cap r_k I^k} >t_k\Big)/\d y \\
&= \frac{K_0^{\ssup Z}}{\int f^{\ssup Z}} \e^{-t_k r_k^{-2}\lambda^{\ssup Z}} f^{\ssup Z}(x/r_k)h_Z(y) ( 1+ o(1)), \\
\end{aligned}
$$
and
$$
\mathbb{P}_x\Big(\tau_{W_Z\cap r_kI^k} >t_k\Big) 
= \frac{K_0^{\ssup Z}}{\int f^{\ssup Z}} \e^{-t_k r_k^{-2}\lambda^{\ssup Z}} f^{\ssup Z}(x/r_k) \int_{W_Z\cap (2I/\pi)^k} h_Z(w)\,\d w (1 + o(1)), 
$$
since the errors $\eps_{t_kr_k^{-2}}$ and $\widetilde\eps_{t_kr_k^{-2}}$ vanish, by our assumption that $\sup_{k\in\N}\gamma(\frac{t_k}{r_k^2})<0$; see \eqref{uniferbou}.
Now a division yields the claim \eqref{claim}.

We now apply \cite[Thm.~3.1]{Eichelsb}, which contains the large-deviation principle for the empirical measure of a random vector with density given by the right-hand side of \eqref{claim} with rate function given in \eqref{ratefunction1} resp.~\eqref{ratefunction2}. Our case Z $=$ A refers to the choice $\Sigma= [-1,1], p(k)=k, w_k\equiv 1, \gamma=1, \beta=1, \kappa=1$ in \cite[Thm.~3.1]{Eichelsb}, and in the case Z $=$ C, one picks $\Sigma= [0,1], p(k)=k, w_k(x)\equiv x, \gamma=2, \beta=1, \kappa=1$. By \eqref{claim}, the empirical measure of a vector having density given by the left-hand side of \eqref{claim}, also satisfies that principle. But this is our assertion.
\end{proof}

We use the large-deviation principle to derive a law of large numbers in the spirit of Wigner's semi-circle law. Let us introduce the following measures $\mu_A$ and $\mu_C$.
\begin{eqnarray}
\mu_A(\d x) &=& \frac{1}{\pi \sqrt{1-x^2}} \,\d x,\qquad x\in [-1,1],\\
\mu_C(\d x) &=& \frac{3}{2\pi x} \sqrt{ \frac{x - 1/9}{1 -x} }\,\d x,\qquad x\in [1/9,1].
\end{eqnarray}
Then $\mu_A$ is the well-known arcsine law.

\begin{corollary}[Law of large numbers]\label{Cor-LLN} Let the situation of Theorem~\ref{thm-LDP} be given. Let $Z$ be in $\{A,C\}$. Then the conditional distribution of $\mu^{\ssup k}_{r_k,t_k}$ under $\P_x( \cdot\,|\, B^{\ssup k}_{[0,t_k]}\subset W_Z\cap r_k I^k)$ converges, uniformly in $x\in W_Z \cap r_k I^k$, weakly towards $\mu_Z$.
\end{corollary}

\begin{proof} That $\mu_A$ and $\mu_C$ are the unique minimizers of $R_A$ and $R_C$, respectively, is well-known from the theory of logarithmic potentials with external fields, see \cite[Ch.~I, Section 1.1; Ch.~IV, Example 5.3]{Saff}. Hence we can apply \cite[Cor.~3.2]{Eichelsb}: using the upper bound of the large-deviation principle one obtains the strong law by applying Borel-Cantelli's lemma, see \cite[B3, Thm.~II]{Ellis}.
\end{proof}







\bigskip

{\sc Patrick Schmid}, Universit\"at Leipzig, Mathematisches Institut, Postfach
100920, D-04009 Leipzig, Germany,
\newline
{\tt Patrick.Schmid@math.uni-leipzig.de  }

\medskip

{\sc Wolfgang K\"onig}, Technical University Berlin, Str. des 17. Juni 136,
10623 Berlin, and Weierstra\ss\ Institute for Applied Analysis and Stochastics,
Mohrenstr. 39, 10117 Berlin, Germany
\newline
{\tt koenig@math.tu-berlin.de, koenig@wias-berlin.de}

\end{document}